\documentclass{amsart}
\usepackage{amsfonts}

\newtheorem{theorem}{Theorem}
\theoremstyle{plain}

\newtheorem{corollary}{Corollary}

\newtheorem{example}{Example}

\newtheorem{lemma}{Lemma}

\newtheorem{proposition}{Proposition}
\newtheorem{remark}{Remark}

\numberwithin{equation}{section}
\input{tcilatex}

\begin{document}
\title[Connection coefficients and orthogonalizing measures]{On affinity
relating two positive measures and the connection coefficients between
polynomials orthogonalized by these measures}
\author{Pawe\l\ J. Szab\l owski}
\address{Department of Mathematics and Information Sciences,\\
Warsaw University of Technology\\
pl. Politechniki 1, 00-661 Warsaw, Poland}
\email{pawel.szablowski@gmail.com}
\date{January 30, 2012}
\thanks{The author is grateful to the unknown referee for drawing his
attention to valuable works of P. Maroni.}

\begin{abstract}
We consider two positive, normalized measures $dA\left( x\right) $ and $%
dB\left( x\right) $ related by the relationship $dA\left( x\right) =\frac{C}{%
x+D}dB\left( x\right) $ or by $dA\left( x\right) \allowbreak =\allowbreak 
\frac{C}{x^{2}+E}dB\left( x\right) $ and $dB\left( x\right) $ is symmetric.
We show that then the polynomial sequences $\left\{ a_{n}\left( x\right)
\right\} ,$ $\left\{ b_{n}\left( x\right) \right\} $ orthogonal with respect
to these measures are related by the relationship $a_{n}\left( x\right)
=b_{n}\left( x\right) +\kappa _{n}b_{n-1}\left( x\right) $ or by $%
a_{n}\left( x\right) \allowbreak =\allowbreak b_{n}\left( x\right)
\allowbreak +\allowbreak \lambda _{n}b_{n-2}\left( x\right) $ for some
sequences $\left\{ \kappa _{n}\right\} $ and $\left\{ \lambda _{n}\right\} .$
We present several examples illustrating this fact and also present some
attempts for extensions and generalizations. We also give some universal
identities involving polynomials $\left\{ b_{n}\left( x\right) \right\} $
and the sequence $\left\{ \kappa _{n}\right\} $ that have a form of Fourier
series expansion of the Radon--Nikodym derivative of one measure with
respect to the other.
\end{abstract}

\maketitle

\section{Introduction}

We study relationship between the pair of orthogonal polynomials and the
pair of measures that make these polynomials orthogonal. This problem has
practical importance. If solved in full generality would enable quick and
easy way of finding sets of orthogonal polynomials for a given measure
simplifying the usual path of the Gram-Smith orthogonalization. Besides it
would provide quick and easy way of finding 'connection coefficients'
between the two analyzed sets of orthogonal polynomials. On its side
'connection coefficients', as it is well known supply many useful
informations about the properties of the involved sets of polynomials. So
far in the literature devoted to connection coefficients like \cite{Szwarc92}%
, \cite{Dimitrov01}, \cite{Area 04} the authors studied the properties of
these coefficients and their relationship to zeros of orthogonal polynomials
in question without referring to the properties of orthogonalizing measures.

We are solving the problem of affinity between connection coefficients and
measures that make polynomials orthogonal only partially. There are still
many challenging questions that we pose in Section \ref{ext} and which are
unsolved to our knowledge.

To be more precise we will assume throughout the paper the following setting:

We consider two sequences of monic, orthogonal polynomials $\left\{
a_{n}\right\} $ and $\left\{ b_{n}\right\} $ such that their 3-term
recurrences are as given below: 
\begin{eqnarray}
a_{n+1}\left( x\right) \allowbreak &=&\allowbreak (x-\alpha _{n})a_{n}\left(
x\right) -\hat{\alpha}_{n-1}a_{n-1}\left( x\right) ,  \label{aa} \\
b_{n+1}\left( x\right) &=&\left( x-\beta _{n}\right) b_{n}\left( x\right) -%
\hat{\beta}_{n-1}b_{n-1}\left( x\right)  \label{bb}
\end{eqnarray}%
with $a_{-1}\left( x\right) \allowbreak =\allowbreak b_{-1}\left( x\right)
\allowbreak =\allowbreak 0,$ $a_{0}\left( x\right) \allowbreak =\allowbreak
b_{0}\left( x\right) \allowbreak =\allowbreak 1.$

In \cite{Szab2010}, Proposition 1 it was shown that if these measures are
such that $\limfunc{supp}A=\allowbreak \limfunc{supp}B$ and $dA\left(
x\right) \allowbreak =\allowbreak \frac{1}{P_{r}\left( x\right) }dB\left(
x\right) $, where $P_{r}$ is a polynomial of order $r,$ then there exist $r$
sequences $\left\{ c_{n}^{\left( j\right) }\right\} _{n\geq 1,1\leq j\leq r}$
such that%
\begin{equation}
a_{n}\left( x\right) \allowbreak =\allowbreak b_{n}\left( x\right)
+\sum_{j=1}^{r}c_{n}^{\left( j\right) }b_{n-j}\left( x\right) .  \label{qo}
\end{equation}%
In the cited result it was not presented how to relate 3-term recurrence
satisfied by say the set $\left\{ b_{n}\right\} $ and the form of the
polynomial $P_{r}$ to the form of the coefficients $\left\{ c_{n}^{\left(
j\right) }\right\} _{n\geq 1,1\leq j\leq r}.$

\begin{remark}
Notice that our assumptions mean in fact that $dA<<dB$ and $\frac{dA(x)}{%
dB(s)}\allowbreak =\allowbreak \frac{1}{P_{r}\left( x\right) },$ where $%
\frac{dA(x)}{dB(s)}$ denotes Radon--Nikodym derivative of $dA$ with respect
to $dB.$
\end{remark}

Relationship like (\ref{qo}) between sets of orthogonal polynomials is
called quasi-orthogonality as defined in \cite{Chih57} and \cite{Dic61}.
More precisely polynomials $\left\{ a_{n}\right\} $ that are related to
polynomials $\left\{ b_{n}\right\} $ by (\ref{qo}) are called
quasi-orthogonal provided $b_{n}^{\prime }s$ are orthogonal. Thus our
problem can be expressed in the following way: If the measures $dA$ and $dB$
are related by $dA\left( x\right) \allowbreak =\allowbreak \frac{1}{%
P_{r}\left( x\right) }dB\left( x\right) ,$ then there exists $r$ sequences
of numbers $\left\{ c_{n}^{\left( j\right) }\right\} _{n\geq 1,1\leq j\leq
r} $ such that quasi-orthogonal polynomials defined by (\ref{qo}) are
orthogonal (with respect to $dA).$

There exits also another, similar in a way, path of research followed by
Pascal Maroni and his associates. The results of their research were
presented in the series of papers \cite{M1}-\cite{M6}. The problem
considered by Maroni concerns general linear regular (i.e. possessing sets
of orthogonal polynomials) functionals $u$ and $v$ (not necessarily in the
form of measure) defined on the linear space of polynomials and related to
one another by the relationship 
\begin{equation*}
x^{m}u=\lambda v,
\end{equation*}%
where $m$ is a fixed integer (in Maroni's papers $m\leq 4)$ and $\lambda $ a
fixed complex number. In majority of cases he assumes that regular $v$ has a
form $<v,p>\allowbreak =\allowbreak \int V(x)p(x)dx$ where $p$ is a
polynomial while $V$ is a locally integrable function rapidly decaying at
infinity. Maroni is interested in conditions for the existence of regular $u$
and also in the relationship between the sets of polynomials orthogonal with
respect to $v$ and $u.$ In his studies he showed that orthogonal polynomials
of $u$ must be linear combinations of the last $m+1$ (i.e. are
quasi-orthogonal). He also obtained some (mostly in the case of $%
m\allowbreak =\allowbreak 1)$ recursive relations relating sets of
orthogonal polynomials of $u$ and $v$ and coefficients of the
quasi-orthogonality.

Notice that even if $v$ is a measure $u$ may be not. Moreover it is
expressed (as it follows from Maroni's papers) by derivatives of Dirac's
delta and Cauchy's principal value. The existence conditions are not simple
and constitute major part of these papers.

That is why although some of the results obtained below were first
discovered by Maroni we will repeat them for the sake of uniformity of
treatment. Of course we will point out which of them were mentioned in
Maroni's papers.

In the present paper we continue the research started in \cite{Szab2010} and
relate the 3-term recurrence coefficients satisfied by $\left\{
b_{n}\right\} $ and the exact form of the polynomial $P_{r}$ for $r=1,2$ to
coefficients $c_{n}^{\left( 1\right) }$ for $r\allowbreak =\allowbreak 1$
and $c_{n}^{\left( 1\right) },$ $c_{n}^{\left( 2\right) }$ for $r\allowbreak
=\allowbreak 2$. We also give the 3-term recurrence coefficients of the
polynomials $\left\{ a_{n}\right\} $. Besides we also provide certain
universal identities satisfied by the polynomials $\left\{ b_{n}\right\} ,$
coefficients $\left\{ \hat{\beta}_{n}\right\} ,$ $\left\{ c_{n}^{\left(
1\right) }\right\} $ and the parameters of the polynomials $P_{r}.$

The paper is organized as follows: In the next Section \ref{gl} we present
our main result concerning the case $r\allowbreak =\allowbreak 1$ and
illustrate it by $3$ examples concerning well known families of polynomials
like Jacobi or Charlier. Examples are presented in Section \ref{przyklady}.
Less complete or less simple and nice results are presented in Section \ref%
{ext}. Here also we will illustrate the developed ideas by a few examples.
Finally Section \ref{dowody} contains less interesting or lengthy proofs of
our results.

\section{Main results\label{gl}}

The simplest but also the most important case is when $r\allowbreak
=\allowbreak 1.$ This case is treated by the theorem below:

\begin{theorem}
\label{glowny}Let the sequence of monic, orthogonal polynomials $\left\{
b_{n}\right\} $ be defined by the 3-term recurrence (\ref{bb}). Suppose that 
$dB\left( x\right) $ is the positive measure that makes these polynomials
orthogonal. Let us consider another normalized measure $dA\left( x\right) $
related to $dB\left( x\right) $ by the relationship: 
\begin{equation}
dA\left( x\right) \allowbreak =\allowbreak \frac{C}{x+D}dB\left( x\right) ,
\label{jednomian}
\end{equation}%
so that $\frac{C}{\left( x+D\right) }\geq 0$ on the support of $dB$ (and of $%
dA$).

Then there exists a number sequence $\left\{ \kappa _{n}\right\} $ defined
by the relationship:%
\begin{equation}
\kappa _{n}=\beta _{n-1}-\frac{\hat{\beta}_{n-2}}{\kappa _{n-1}}+D,
\label{_1}
\end{equation}%
$n\geq 2$ with $\kappa _{1}\allowbreak =\allowbreak \beta _{0}+D-C,$ such
that the sequence of monic polynomials defined by: 
\begin{equation}
a_{n}\left( x\right) =b_{n}\left( x\right) +\kappa _{n}b_{n-1}\left(
x\right) .  \label{zaleznosc}
\end{equation}%
satisfies the 3-term recurrence (\ref{aa}) with:%
\begin{eqnarray}
\alpha _{n}\allowbreak  &=&\allowbreak \beta _{n}+\kappa _{n}-\kappa _{n+1},
\label{new1} \\
\hat{\alpha}_{n-1} &=&\kappa _{n}\frac{\hat{\beta}_{n-2}}{\kappa _{n-1}},
\label{new2}
\end{eqnarray}%
and is orthogonal with respect to the measure $dA\left( x\right) .$
\end{theorem}

\begin{remark}
Recursive equations (\ref{_1}) and (\ref{new1}) were obtained by P. Maroni
in \cite{M1} as sidelines of his results obtained in a slightly different
context. (\ref{new2}) was obtained in a different but equivalent form.
\end{remark}

\begin{proof}
Is shifted to section \ref{dowody}.
\end{proof}

We have immediate remarks, observations and corollaries

\begin{remark}
All coefficients $\kappa _{n}$ have the same sign i.e. are either all
positive or all negative. This follows the fact that since $dA\left(
x\right) $ and $dB\left( x\right) $ are positive measures we must have
nonnegative both $\hat{\alpha}_{n}$ and $\hat{\beta}_{n}.$ Then we use (\ref%
{new2}).
\end{remark}

\begin{remark}
Notice that following relationship $\kappa _{1}\allowbreak =\allowbreak
\beta _{0}+D-C$ and (\ref{zaleznosc}), $\frac{C}{x+D}$ can be written as $%
\frac{1}{a_{1}\left( x\right) /C+1}$ which fits assumptions of Proposition 1
of \cite{Szab2010}.
\end{remark}

\begin{remark}
Following (\ref{new2}) and the fact that $\hat{\alpha}_{n-1}\geq 0$ we
deduce that either $\beta _{n}\allowbreak +\allowbreak D\allowbreak \geq
\allowbreak 0$ or $\beta _{n}\allowbreak +\allowbreak D\allowbreak \leq
\allowbreak 0$ for all $n\geq 1.$ Consequently either we have for all $n\geq
0$ 
\begin{equation*}
\beta _{n-1}+D\leq \kappa _{n}\leq 0~~\text{or~~}0\leq \kappa _{n}\leq \beta
_{n-1}+D.
\end{equation*}
\end{remark}

\begin{corollary}
\label{rozw}Under assumptions of Theorem \ref{glowny} we have 
\begin{equation}
b_{n}\left( x\right) =a_{n}\left( x\right) +\sum_{j=1}^{n}\left( -1\right)
^{j}\left( \prod_{k=n-j+1}^{n}\kappa _{k}\right) a_{n-j}\left( x\right) 
\label{odwr}
\end{equation}%
for $n\allowbreak =\allowbreak 0,1,2,\ldots $ . Further under additional
assumption that $\int_{\limfunc{supp}B}\frac{1}{(x+D)^{2}}dB\left( x\right)
<\infty $ we have:%
\begin{equation*}
1+\sum_{n\geq 1}\left( \prod_{k=1}^{n}\frac{\kappa _{k}}{\hat{\beta}_{k-1}}%
\right) ^{2}=C^{2}\int_{\limfunc{supp}B}\frac{1}{(x+D)^{2}}dB\left( x\right) 
\end{equation*}%
and 
\begin{equation}
\frac{C}{x+D}=1+\sum_{n\geq 1}\left( -1\right) ^{n}\left( \prod_{k=1}^{n}%
\frac{\kappa _{k}}{\hat{\beta}_{k-1}}\right) b_{n}\left( x\right) ,
\label{szereg}
\end{equation}%
on $\limfunc{supp}B$ in $L_{2}\left( \limfunc{supp}B,\mathcal{B},dB\left(
x\right) \right) .$ If additionally $\sum_{n\geq 1}\left( \prod_{k=1}^{n}%
\frac{\kappa _{k}}{\hat{\beta}_{k-1}}\right) ^{2}\log ^{2}n<\infty ,$ then
convergence in (\ref{szereg}) is almost ($dB\left( x\right) )$ pointwise on $%
\limfunc{supp}B.$
\end{corollary}

\begin{proof}
Using (\ref{odwr}) we have $b_{n}\left( x\right) \allowbreak +\allowbreak
\kappa _{n}b_{n-1}\left( x\right) \allowbreak =\allowbreak
a_{n}+\sum_{j=1}^{n}\left( -1\right) ^{j}\left( \prod_{k=n-j+1}^{n}\kappa
_{k}\right) a_{n-j}\left( x\right) \allowbreak +\allowbreak \kappa
_{n}a_{n-1}\allowbreak +\allowbreak \sum_{j=1}^{n-1}\left( -1\right)
^{j}\left( \prod_{k=n-j}^{n}\kappa _{k}\right) a_{n-1-j}\left( x\right)
\allowbreak =\allowbreak a_{n}.$ Now we apply idea of ratio of density
expansion presented in \cite{Szab2010} and use (\ref{odwr}). On the way we
notice that the requirement that both measures (i.e. $dA$ and $dB)$ have
densities with respect to Lebesgue measure can be dropped. We also utilize
the fact that $\int_{\limfunc{supp}B}b_{n}^{2}\left( x\right) dB\left(
x\right) \allowbreak =\allowbreak \prod_{k=0}^{n-1}\hat{\beta}_{k}$ which
follows Favard's Theorem. The we also use Rademacher--Menshov theorem
concerning almost sure convergence of $L_{2}$ converging Fourier series.
\end{proof}

\section{Examples\label{przyklady}}

To illustrate how simple and easy is to utilize the presented in the
previous section observations and rules let us consider the following few
examples:

It should be remarked that although examples concerning Jacobi and Legendre
polynomials were considered by Maroni in \cite{M1} they were illustrating
different phenomena discovered by Maroni in this paper. In particular
formulae (\ref{Jkappa}), (\ref{Jexp}) and (\ref{expJ}) do not appear in
Maroni's paper. We present them in order to illustrate the use of equations (%
\ref{_1}), (\ref{new1}), (\ref{new2}) and expansion (\ref{szereg}).

\begin{example}[Jacobi polynomials]
Recall (e.g. basing on \cite{Andrews1999} or \cite{IA}) that monic Jacobi
polynomials $J_{n}^{\left( \alpha ,\beta \right) }\left( x\right) $ satisfy
the following 3-term recurrence: 
\begin{subequations}
\label{Jacobi}
\begin{gather}
J_{n+1}^{(\alpha ,\gamma )}(z)=\left( x+\frac{\alpha ^{2}-\gamma ^{2}}{%
(2n+\alpha +\gamma +2)(\alpha +\gamma +2n)}\right) J_{n}^{\left( \alpha
,\gamma \right) }\left( x\right)   \label{Jacobi1} \\
-\frac{4n\left( \alpha +\gamma +n\right) (n+\alpha )(n+\gamma )}{\left(
\alpha +\gamma +2n-1\right) (2n+\alpha +\gamma )^{2}(\alpha +\gamma +2n+1)}%
J_{n-1}^{(\alpha ,\gamma )}(z).  \label{Jacobi2}
\end{gather}%
Besides one knows also that the normalized measure that makes these
polynomials orthogonal is the following: 
\end{subequations}
\begin{equation*}
f\left( x;\alpha ,\gamma \right) =\frac{\Gamma \left( \alpha +\gamma
+2\right) }{2^{\alpha +\gamma +1}\Gamma \left( \gamma +1\right) \Gamma
\left( \alpha +1\right) }\left( 1-x\right) ^{\alpha }(1+x)^{\gamma },
\end{equation*}%
where $\Gamma \left( \eta \right) $ denotes value of the Gamma function at $%
\eta ,$ for $\left\vert x\right\vert <1$ and $\alpha ,\gamma >-1.$

Now let us take $\alpha >0,\gamma >-1,$ $dB\left( x\right) \allowbreak
=\allowbreak f\left( x;\alpha ,\gamma \right) dx$ and $dA\left( x\right)
\allowbreak =\allowbreak f\left( \alpha -1,\gamma \right) dx,$ $b_{n}\left(
x\right) \allowbreak =\allowbreak J_{n}^{\left( \alpha ,\gamma \right)
}\left( x\right) $ and $a_{n}\left( x\right) \allowbreak =\allowbreak
J_{n}^{\left( \alpha -1,\gamma \right) }\left( x\right) .$ One can easily
notice that $dA\left( x\right) \allowbreak =\allowbreak \frac{C}{(-1+x)}%
dB\left( x\right) ,$ where $C=\allowbreak -\frac{2\alpha }{\alpha +\gamma +1}%
,$ hence $D\allowbreak =\allowbreak -1.$ From (\ref{Jacobi}) it follows also
that 
\begin{eqnarray}
\beta _{n}\allowbreak  &=&\allowbreak -\frac{\alpha ^{2}-\gamma ^{2}}{%
(2n+\alpha +\gamma +2)(\alpha +\gamma +2n)},  \label{_bJ} \\
\hat{\beta}_{n-1} &=&\frac{4n\left( \alpha +\gamma +n\right) (n+\alpha
)(n+\gamma )}{\left( \alpha +\gamma +2n-1\right) (2n+\alpha +\gamma
)^{2}(\alpha +\gamma +2n+1)}.  \label{_bbJ}
\end{eqnarray}%
Thus $\kappa _{1}\allowbreak =\allowbreak \beta _{0}+D-C\allowbreak =-\frac{%
\alpha ^{2}-\gamma ^{2}}{(\alpha +\gamma +2)(\alpha +\gamma )}\allowbreak +%
\frac{2\alpha }{\alpha +\gamma +1}-1\allowbreak \allowbreak =\allowbreak -2%
\frac{\gamma +1}{\left( \alpha +\gamma +1\right) \left( \alpha +\gamma
+2\right) }$ and consequently coefficients $\kappa _{n}$ satisfy recursive
equation:%
\begin{equation*}
\kappa _{n}=\beta _{n-1}-1-\frac{\hat{\beta}_{n-2}}{\kappa _{n-1}},
\end{equation*}%
for $n\geq 2.$ One can also easily notice that 
\begin{equation}
\kappa _{n}=-\frac{2n(n+\gamma )}{\left( \alpha +\gamma +2n\right) \left(
\alpha +\gamma +2n-1\right) }  \label{Jkappa}
\end{equation}%
satisfies above mentioned recursive equation. Hence 
\begin{equation}
J_{n}^{\left( \alpha -1,\gamma \right) }\left( x\allowbreak \right)
=J_{n}^{\left( \alpha ,\gamma \right) }\left( x\right) +\kappa
_{n}J_{n-1}^{\left( \alpha ,\gamma \right) }(x).  \label{Jexp}
\end{equation}

As far as application of Corollary \ref{rozw} is concerned we have the
following identity true for $\alpha >1,$ $\gamma >-1,$ and almost all $%
\left\vert x\right\vert <1:$ 
\begin{equation}
1=\frac{\alpha +\gamma +1}{2\alpha }(1-x)(1+\sum_{n\geq 1}^{\infty }\frac{%
(\alpha +\gamma +1)_{2n}}{2^{n}(\alpha +\gamma +1)\left( \alpha +1\right)
_{n}\left( \alpha +\gamma +1\right) _{n}}J_{n}^{\left( \alpha ,\gamma
\right) }\left( x\right) ),  \label{expJ}
\end{equation}
where we use the so called Pochhammer symbol $\left( a\right)
_{n}\allowbreak =\allowbreak a\left( a+1\right) \ldots (a+n-1).$ This is so
since $\frac{\kappa _{n}}{\gamma _{n-1}}\allowbreak =\allowbreak -\frac{%
\left( \alpha +\gamma +2n+1\right) (\alpha +\gamma +2n)}{2\left( \alpha
+n\right) \left( \alpha +\gamma +n\right) },$ by (\ref{_bbJ}) and (\ref%
{Jkappa}) and because $\int_{-1}^{1}\frac{1}{(1-x)^{2}}dB\left( x\right)
<\infty $ for $\gamma >-1,$ $\alpha -2>-1.$
\end{example}

\begin{example}[Charlier polynomials]
Basing on \cite{Koek} let us recall that monic Charlier polynomials $\left\{
c_{n}\left( x;\lambda \right) \right\} _{n\geq -1}$ are polynomials given by
the following 3-term recurrence%
\begin{equation*}
c_{n+1}\left( x;\lambda \right) =(x-n-\lambda )c_{n}\left( x;\lambda \right)
-n\lambda c_{n-1}\left( x;\lambda \right) ,
\end{equation*}%
with $c_{-1}\left( x;\lambda \right) \allowbreak =\allowbreak 0,$ $%
c_{0}\left( x;\lambda \right) \allowbreak =\allowbreak 1.$ For $\lambda >0$
they are orthogonal with respect to discrete measure concentrated at
nonnegative integers with mass at $n$ equal to $\exp \left( -\lambda \right) 
\frac{\lambda ^{n}}{n!},$ $n\geq 0.$ Another words this measure is the
Poisson normalized measure.

In order not to complicate too much let us take $dB\left( n\right)
\allowbreak =\allowbreak $ $\exp \left( -\lambda \right) \frac{\lambda ^{n}}{%
n!}$ and $dA\left( n\right) \allowbreak =\allowbreak \frac{C}{n+1}dB\left(
n\right) $ for $n\allowbreak =\allowbreak 0,1,\ldots $ . Since $\sum_{n\geq
0}^{\infty }\frac{\lambda ^{n}}{\left( n+1\right) !}\allowbreak =\allowbreak 
\frac{\left( \exp \left( \lambda \right) -1\right) }{\lambda }$ we see that $%
C\allowbreak =\allowbreak \frac{\lambda \exp \left( \lambda \right) }{\exp
\left( \lambda \right) -1}.$ Naturally we have also $D=1$ and $\beta
_{n}\allowbreak =\allowbreak n+\lambda $ and $\hat{\beta}_{n-1}\allowbreak
=\allowbreak n\lambda ,$ hence $\kappa _{1}\allowbreak =\allowbreak \lambda
+1-\frac{\lambda \exp \left( \lambda \right) }{-1+\exp \left( \lambda
\right) }\allowbreak =\allowbreak $ $\frac{\exp \left( \lambda \right)
-1-\lambda }{\exp (\lambda )-1}.$ Thus recursive equation satisfied by
coefficients $\kappa _{n}$ is the following:%
\begin{equation*}
\kappa _{n}=n+\lambda -\frac{\left( n-1\right) \lambda }{\kappa _{n-1}},
\end{equation*}%
$n\geq 2.$ In particular we have $\kappa _{2}\allowbreak =\allowbreak 2\frac{%
\exp \left( \lambda \right) -1-\lambda -\lambda ^{2}/2}{\exp \left( \lambda
\right) -1-\lambda },$ $\kappa _{3}\allowbreak =\allowbreak 3\frac{\exp
\left( \lambda \right) -1-\lambda -\lambda ^{2}/2-\lambda ^{3}/3!}{\exp
\left( \lambda \right) -1-\lambda -\lambda ^{2}/2}$ and in general it is
easy to see that 
\begin{equation*}
\kappa _{n}\allowbreak =\allowbreak n\frac{\exp \left( \lambda \right)
-\sum_{j=0}^{n}\frac{\lambda ^{j}}{j!}}{\exp \left( \lambda \right)
-\sum_{j=0}^{n-1}\frac{\lambda ^{j}}{j!}}=n\frac{\sum_{j\geq n+1}\frac{%
\lambda ^{j}}{j!}}{\sum_{j\geq n}\frac{\lambda ^{j}}{j!}}.
\end{equation*}%
Thus we have in particular 
\begin{eqnarray*}
a_{n}\left( x\right) \allowbreak  &=&\allowbreak c_{n}\left( x\right)
+\kappa _{n}c_{n-1}\left( x\right) , \\
\alpha _{n} &=&n+\lambda +\kappa _{n}-\kappa _{n+1}, \\
\hat{\alpha}_{n-1} &=&\frac{\lambda n(\sum_{j\geq n+1}\frac{\lambda ^{j}}{j!}%
)\left( \sum_{j\geq n-1}\frac{\lambda ^{j}}{j!}\right) }{(\sum_{j\geq n}%
\frac{\lambda ^{j}}{j!})^{2}}.
\end{eqnarray*}%
As the application of Corollary \ref{rozw} we have the following identity
true for $\lambda >0$ and $x\allowbreak =\allowbreak 0,1,\ldots $%
\begin{equation}
e^{\lambda }=(1+x)(\frac{e^{\lambda }-1}{\lambda }+\sum_{n\geq 1}\left(
-1\right) ^{n}c_{n}\left( x;\lambda \right) \sum_{k\geq n+1}\frac{\lambda
^{k-n-1}}{k!}).  \label{CA}
\end{equation}%
This is so since $\frac{\kappa _{n}}{\hat{\beta}_{n-1}}=\frac{\exp \left(
\lambda \right) -\sum_{j=0}^{n}\frac{\lambda ^{j}}{j!}}{\lambda (\exp \left(
\lambda \right) -\sum_{j=0}^{n-1}\frac{\lambda ^{j}}{j!})}$ and consequently 
$\prod_{k=1}^{n}\frac{\kappa _{k}}{\hat{\beta}_{k-1}}\allowbreak
=\allowbreak \frac{\exp \left( \lambda \right) -\sum_{k=0}^{n}\frac{\lambda
^{k}}{k!}}{\lambda ^{n}}\allowbreak =\allowbreak \sum_{k\geq n+1}\frac{%
\lambda ^{k-n}}{k!}.$ Let us observe that (\ref{CA}) is not satisfied for
non-positive integer $x.$
\end{example}

\begin{example}[Legendre polynomials]
As it is known Legendre polynomials are the Jacobi polynomials with $\alpha
,\gamma \allowbreak =\allowbreak 0$. As $dB\left( x\right) $ let us consider
measure with the density $f(x;0,0)\allowbreak =\allowbreak 1/2$ on $[-1,1].$
As $dA\left( x\right) $ let us consider measure with the density $\frac{C}{%
2(3-x)}$ on $[-1,1].$ Parameter $C$ we get by direct integration, namely $%
C\allowbreak =\allowbreak \frac{-2}{\ln 2}$ while $D\allowbreak =\allowbreak
\allowbreak -3.$ Further using (\ref{_bJ}) and (\ref{_bbJ}) we get: 
\begin{equation*}
\beta _{n}\allowbreak =\allowbreak 0,\hat{\beta}_{n-1}\allowbreak =\frac{%
n^{2}}{(2n-1)(2n+1)}.
\end{equation*}%
Hence $\kappa _{1}\allowbreak =\allowbreak 0-3+\frac{2}{\ln 2}$ and
consequently coefficients $\kappa _{n}$ are given by the following recursive
equation:%
\begin{equation*}
\kappa _{n+1}\allowbreak =\allowbreak -3-\frac{n^{2}}{(2n-1)(2n+1)\kappa _{n}%
},
\end{equation*}%
for $n\geq 1.$ In particular we get $\kappa _{2}\allowbreak =\allowbreak -3-%
\frac{1}{3(-3+2/\ln 2)}\allowbreak =\allowbreak $ $-\frac{\left( 26\ln
2-18\right) }{9\ln 2-6}.$ Finally we deduce that polynomials defined by 
\begin{equation*}
a_{n}\left( x\right) =J_{n}^{\left( 0,0\right) }+\kappa _{n}J_{n-1}^{\left(
0,0\right) }(x),
\end{equation*}%
are orthogonal with respect to the measure with the density: $\frac{2}{%
(3-x)\ln 2}$ on $[-1,1].$
\end{example}

\section{Extensions and open problems\label{ext}}

In this section we are going to present some generalizations of the results
of Section \ref{gl}. The results are not as nice and compact as the ones
presented above that is why we present them here. We will also pose some
open problems that appeared immediately when writing the article.

Let us return to the setting that was presented in the Introduction and
consider the case $r\allowbreak =\allowbreak 2.$ Let us assume that measures 
$dA$ and $dB$ are related to one another by the relationship 
\begin{equation}
dA\left( x\right) =\frac{C}{x^{2}+Dx+E}dB\left( x\right)  \label{dwumian}
\end{equation}%
and that constants $C,D,E$ are such that $\frac{C}{x^{2}+Cx+E}\geq 0$ on $%
\limfunc{supp}B$ and that measure $dA$ is normalized. Following cited
already \cite{Szab2010}, Proposition 1 we deduce that then polynomials $%
\left\{ a_{n}\right\} $ and $\left\{ b_{n}\right\} $ orthogonal with respect
to these measures are related by the relationship 
\begin{equation}
a_{n}\left( x\right) \allowbreak =\allowbreak b_{n}\left( x\right) +\kappa
_{n}b_{n-1}\left( x\right) +\lambda _{n}b_{n-2}\left( x\right) ,  \label{zal}
\end{equation}%
for some number sequences $\left\{ \kappa _{n}\right\} $ and $\left\{
\lambda _{n}\right\} .$ given in the Proposition below:

\begin{proposition}
\label{uog}Suppose normalized, positive measures $dA$ and $dB$ are related
to one another by (\ref{dwumian}). Let further polynomial sequences $\left\{
a_{n}\right\} $ and $\left\{ b_{n}\right\} $ orthogonal with respect to
these measures satisfy respectively 3-term recurrence (\ref{aa}) and (\ref%
{bb}). Then there exist two number sequences $\left\{ \kappa _{n}\right\} $
and $\left\{ \lambda _{n}\right\} $ such that (\ref{zal}) is satisfied.
Moreover number sequences $\left\{ \kappa _{n}\right\} ,$ $\left\{ \lambda
_{n}\right\} $, $\left\{ \alpha _{n}\right\} $ $\left\{ \hat{\alpha}%
_{n}\right\} ,$ $\left\{ \beta _{n}\right\} ,$ $\left\{ \hat{\beta}%
_{n}\right\} $ are related to one another by the system of equations: 
\begin{eqnarray}
\kappa _{n+1}+\alpha _{n} &=&\beta _{n}+\kappa _{n},  \label{s1} \\
\lambda _{n+1}+\alpha _{n}\kappa _{n}+\hat{\alpha}_{n-1} &=&\hat{\beta}%
_{n-1}+\kappa _{n}\beta _{n-1}+\lambda _{n},  \label{s2} \\
\alpha _{n}\lambda _{n}+\hat{\alpha}_{n-1}\kappa _{n-1} &=&\kappa _{n}\hat{%
\beta}_{n-2}+\lambda _{n}\beta _{n-2},  \label{s3} \\
\hat{\alpha}_{n-1}\lambda _{n-1} &=&\lambda _{n}\hat{\beta}_{n-3}.
\label{s4}
\end{eqnarray}%
with $\lambda _{1}\allowbreak =\allowbreak 0$ and $\kappa _{1},$ $\kappa
_{2} $ and $\lambda _{2}$ defined as solutions of the system of $7$
equations $0\allowbreak =\allowbreak \int_{\limfunc{supp}B}\left(
b_{1}\left( x\right) +\kappa _{1}\right) dA\left( x\right) \allowbreak
=\allowbreak \int_{\limfunc{supp}B}\left( b_{2}\left( x\right) +\kappa
_{2}b_{1}\left( x\right) +\lambda _{2}\right) dA\left( x\right) \allowbreak
=\allowbreak \int_{\limfunc{supp}B}\left( b_{2}\left( x\right) +\kappa
_{2}b_{1}\left( x\right) +\lambda _{2}\right) (b_{1}\left( x\right) +\kappa
_{1})dA\left( x\right) \allowbreak =\allowbreak \int_{\limfunc{supp}%
B}b_{1}\left( x\right) \left( x^{2}+Dx+E\right) dA\left( x\right)
\allowbreak =\allowbreak \int_{\limfunc{supp}B}b_{2}\left( x\right) \left(
x^{2}+Dx+E\right) dA\left( x\right) ,$ $\int_{\limfunc{supp}B}\left(
x^{2}+Dx+E\right) dA\left( x\right) \allowbreak =\allowbreak C,$ \newline
$\int_{\limfunc{supp}B}b_{1}^{2}\left( x\right) \left( x^{2}+Dx+E\right)
dA\left( x\right) \allowbreak =\allowbreak \hat{\beta}_{0}$ with $4$
additional unknowns \newline
$\int_{\limfunc{supp}B}b_{1}\left( x\right) dA\left( x\right) ,$ $\int_{%
\limfunc{supp}B}b_{2}\left( x\right) dA\left( x\right) ,$ $\int_{\limfunc{%
supp}B}b_{1}\left( x\right) b_{2}\left( x\right) dA\left( x\right) $, $\int_{%
\limfunc{supp}B}b_{2}^{2}\left( x\right) dA\left( x\right) .$
\end{proposition}

\begin{proof}
Uninteresting proof is shifted to Section \ref{dowody}.
\end{proof}

Visibly it is hard to solve system of equations (\ref{s1})-(\ref{s4}) for $%
\{\kappa _{n},\lambda _{n},\alpha _{n},\hat{\alpha}_{n}\}$ in general. Below
we will present one example where it is simple.

\begin{example}[Kesten--McKay distribution]
This example concerns measure that is called Kesten--McKay. It appeared in
probability in the context of random matrices. One of the particular
examples of its density is the density of the form:%
\begin{equation*}
f\left( x;y,\rho \right) =\frac{(1-\rho ^{2})\sqrt{4-x^{2}}}{2\pi ((1-\rho
^{2})^{2}-\rho (1+\rho ^{2})xy+\rho ^{2}(x^{2}+y^{2}))},
\end{equation*}%
for $\left\vert x\right\vert ,\left\vert y\right\vert \leq 2,$ $\rho ^{2}<1.$
To see that $\int_{-2}^{2}f\left( x;y,\rho \right) dx\allowbreak
=\allowbreak 1$ for all $\left\vert y\right\vert \leq 2$ $and$ $\rho ^{2}<1$
is difficult hence to obtain sequence of polynomials orthogonal with respect
to it by Gram-Schmidt procedure is quite hard. As one can easily see this
density is a particular example of the relationship (\ref{dwumian}) with $%
dB(x)\allowbreak =\allowbreak \frac{1}{2\pi }\sqrt{4-x^{2}}dx,$ $%
b_{n}(x)\allowbreak =\allowbreak U_{n}(x/2),$ where $U_{n}(x)$ are the
Chebyshev polynomials of the second kind (for details see e.g. \cite%
{Andrews1999})). One can notice that polynomials $b_{n}$ satisfy the
following 3-term recurrence:%
\begin{equation*}
b_{n+1}(x)=xb_{n}(x)-b_{n-1}(x),
\end{equation*}%
with $b_{-1}(x)\allowbreak =\allowbreak 0,$ $b_{0}(x)\allowbreak
=\allowbreak 1.$ First of all notice that if $\rho \allowbreak =\allowbreak 0
$ then we deal with trivial case. Hence let us assume that $0<\left\vert
\rho \right\vert <1.$ We have $\beta _{n}\allowbreak =\allowbreak 0$ and $%
\hat{\beta}_{n}\allowbreak =\allowbreak 1$ and further $C\allowbreak
=\allowbreak \frac{(1-\rho ^{2})}{\rho ^{2}},$ $D\allowbreak =\allowbreak -%
\frac{(1+\rho ^{2})y}{\rho },$ $E\allowbreak =\allowbreak (\frac{1-\rho ^{2}%
}{\rho })^{2}+y^{2}.$ By direct computation we check that $\kappa
_{1}\allowbreak =\allowbreak \kappa _{2}\allowbreak =\allowbreak -\rho y$
and $\lambda _{2}\allowbreak =\allowbreak \rho ^{2}.$ Now inserting all of
the ingredients to equations (\ref{s1})-(\ref{s4}) we see that $\kappa
_{n}\allowbreak =\allowbreak -\rho y$ for $n\geq 1,$ $\lambda
_{n}\allowbreak =\allowbreak \rho ^{2}$ for all $n\geq 2$ , $\alpha
_{n}\allowbreak =\allowbreak 0,$ and $\hat{\alpha}_{n}\allowbreak
=\allowbreak 1$ for all $n\geq 1.$ Thus $a_{n}(x)\allowbreak =\allowbreak
b_{n}(x)-\rho yb_{n-1}(x)+\rho ^{2}b_{n-2}(x)$ for all $n\geq 2.$
\end{example}

Now let us simplify calculations by assuming that the measure $dB$ and the
polynomial $\left( x^{2}+Dx+E\right) $ are symmetric which implies that
polynomials orthogonal with respect to $dB$ (i.e. $b_{n})$ must contain only
either even or odd powers of $x$. Hence coefficients $\beta _{n}$ are equal
to zero for $n\geq 0$ and also that $D\allowbreak =\allowbreak 0.$
Consequently measure $dA$ must also be symmetric and by similar argument we
deduce that coefficients $\alpha _{n}=0$ for $n\geq 0.$ This results in the
fact that coefficients $\kappa _{n}$ are also zero for all $n\geq 0.$

As a result we have the following Lemma which is in fact a corollary of the
Proposition \ref{uog}.

\begin{lemma}
Suppose normalized, positive measures $dA$ and $dB$ are related to one
another by the relationship: 
\begin{equation*}
dA\left( x\right) \allowbreak =\allowbreak \frac{C}{x^{2}+E}dB\left(
x\right) .
\end{equation*}%
Let further respectively polynomial sequences $\left\{ a_{n}\right\} $ and $%
\left\{ b_{n}\right\} $ orthogonal with respect to these measures satisfy
3-term recurrence (\ref{aa}) and (\ref{bb}). With $\beta _{n}\allowbreak
=\allowbreak 0$ for $n\geq 0.$ Then there exist a number sequences $\left\{
\lambda _{n}\right\} $ such that 
\begin{equation}
a_{n}\left( x\right) \allowbreak =\allowbreak b_{n}\left( x\right) +\lambda
_{n}b_{n-2}\left( x\right) ,  \label{_2}
\end{equation}%
and $\alpha _{n}\allowbreak =\allowbreak 0$ for $n\geq 0$. Moreover number
sequence $\left\{ \lambda _{n}\right\} $, satisfies the following second
order recursive equation for $n\geq 3:$ 
\begin{equation}
\lambda _{n+1}=\lambda _{n}+\hat{\beta}_{n-1}-\frac{\lambda _{n}}{\lambda
_{n-1}}\hat{\beta}_{n-3},  \label{eqn2}
\end{equation}%
with $\lambda _{1}\allowbreak =\allowbreak 0,$ $\lambda _{2}\allowbreak
=\allowbreak \hat{\beta}_{0}+E-C,$ $\lambda _{3}=\hat{\beta}_{1}+E-\frac{E}{%
C-E}\hat{\beta}_{0}.$

Coefficients $\hat{\alpha}_{n}$ are given by relationship:%
\begin{equation*}
\hat{\alpha}_{n}=\frac{\lambda _{n+1}}{\lambda _{n}}\hat{\beta}_{n-2}.
\end{equation*}
\end{lemma}

\begin{proof}
We apply assumptions to the system of equations (\ref{s1}-\ref{s4}) getting: 
\begin{eqnarray*}
\lambda _{n+1}+\hat{\alpha}_{n-1} &=&\lambda _{n}+\hat{\beta}_{n-1}, \\
\hat{\alpha}_{n-1}\lambda _{n-1} &=&\lambda _{n}\hat{\beta}_{n-3},
\end{eqnarray*}%
from which we get (\ref{eqn2}). To get initial conditions we notice that $%
a_{2}\left( x\right) \allowbreak =\allowbreak x^{2}-\hat{\beta}_{0}+\lambda
_{2}\allowbreak =\allowbreak x^{2}\allowbreak +\allowbreak E+(\lambda
_{2}-\beta _{0}-E),$ so from the relationships $\int_{\limfunc{supp}%
B}a_{2}\left( x\right) dA\left( x\right) \allowbreak =\allowbreak 0$ and $%
\int_{\limfunc{supp}B}(x^{2}+E)dA\left( x\right) \allowbreak \allowbreak
=\allowbreak C$ we get $C\allowbreak +\allowbreak (\lambda _{2}-\beta
_{0}-E)\allowbreak =\allowbreak 0.$ Now to get $\lambda _{3}$ we use
relationship: $\int_{\limfunc{supp}B}a_{1}\left( x\right) a_{3}\left(
x\right) dA\left( x\right) \allowbreak =\allowbreak 0,$ using on the way the
fact that $a_{3}\left( x\right) \allowbreak =\allowbreak b_{3}\left(
x\right) +\lambda _{3}x\allowbreak =\allowbreak x(x^{2}-\hat{\beta}%
_{0})\allowbreak -\allowbreak \hat{\beta}_{1}x\allowbreak +\allowbreak
\lambda _{3}x\allowbreak =\allowbreak x^{3}\allowbreak +\allowbreak
x(\lambda _{3}-\hat{\beta}_{0}-\hat{\beta}_{1})\allowbreak $ and that $\int_{%
\limfunc{supp}B}(x^{2}-\hat{\beta}_{0})(x^{2}+E)dA\left( x\right)
\allowbreak =\allowbreak 0.$ We have: $0\allowbreak =\allowbreak \int_{%
\limfunc{supp}B}((x^{2}-\hat{\beta}_{0})(x^{2}+E)+(\lambda _{3}-\hat{\beta}%
_{1}-E)x^{2}+\hat{\beta}_{0}E)dA\left( x\right) \allowbreak =\allowbreak
\int_{\limfunc{supp}B}((\lambda _{3}-\hat{\beta}_{1}-E)(x^{2}+E)+\hat{\beta}%
_{0}E-E(\lambda _{3}-\hat{\beta}_{1}-E))dA\left( x\right) \allowbreak
=\allowbreak C(\lambda _{3}-\hat{\beta}_{1}-E)-E(\lambda _{3}-\hat{\beta}%
_{1}-E-\hat{\beta}_{0})=0$.
\end{proof}

We will briefly illustrate this Lemma$\allowbreak $ by the following example.

\begin{example}[Jacobi polynomials revisited]
Let us consider the symmetric case i.e. assuming that parameters $\alpha $
and $\gamma $ are equal say to $a$. Then $\beta _{n}\allowbreak =\allowbreak
0$ and $\hat{\beta}_{n-1}\allowbreak =\allowbreak \frac{n(n+2a)}{%
(2a+2n-1)(2a+2n+1)}.$ Parameters $C$ and $E$ are now equal to $-\frac{2a}{%
2a+1}$ and $-1$ respectively. Hence 
\begin{eqnarray*}
\lambda _{2} &=&\hat{\beta}_{0}+E-C=-\frac{2}{(2a+1)(2a+3)} \\
\lambda _{3} &=&\hat{\beta}_{1}+E-\frac{E}{C-E}\hat{\beta}_{0}=-\frac{6}{%
(2a+3)(2a+5)}.
\end{eqnarray*}%
Besides one can easily check that the sequence $\left\{ -\frac{n(n-1)}{%
(2a+2n-1)(2a+2n-3)}\right\} $ satisfies (\ref{eqn2}). So $\lambda
_{n}\allowbreak =\allowbreak -\frac{n(n-1)}{(2a+2n-1)(2a+2n-3)},$ $n\geq 2.$
Hence we have:%
\begin{equation*}
J_{n}^{\left( \alpha -1,\alpha -1\right) }\left( x\right) =J_{n}^{\left(
\alpha ,\alpha \right) }-\frac{n(n-1)}{(2a+2n-1)(2a+2n-3)}J_{n-2}^{\left(
\alpha ,\alpha \right) }.
\end{equation*}%
In particular notice that for $\alpha \allowbreak =\allowbreak 1/2$ we have $%
J_{n}^{(1/2,1/2)}\left( x\right) \allowbreak =\allowbreak U_{n}\left(
x\right) /2^{n}$ where $U_{n}$ are the Chebyshev polynomials of the second
kind, while $J_{n}^{(-1/2,-1/2)}\left( x\right) \allowbreak =\allowbreak
T_{n}\left( x\right) /2^{n-1}$ for $n\geq 2$ and $J_{n}^{(-1/2,-1/2)}\left(
x\right) \allowbreak =\allowbreak T_{n}\left( x\right) $ for $n\allowbreak
=\allowbreak 0,1,$ where $\allowbreak T_{n}\left( x\right) ,$ are the
Chebyshev polynomials of the first kind. Besides $\left[ -\frac{n(n-1)}{%
(2a+2n-1)(2a+2n-3)}\right] _{\alpha =1/2}\allowbreak =\allowbreak -\frac{1}{4%
}$ and we end up with well known relationship between Chebyshev polynomials
of the first and second kind: 
\begin{equation*}
T_{n}\left( x\right) \allowbreak =\allowbreak (U_{n}(x)-U_{n-2}\left(
x\right) )/2.
\end{equation*}
\end{example}

\begin{remark}
Notice that we could have reached the result in the above mentioned example
by applying the procedure described in Section \ref{gl} twice once for
monomial $1-x$ and then for $1+x.$
\end{remark}

\begin{remark}
Notice also that one could invert relationship (\ref{_2}) and find
connection coefficients of polynomials $\left\{ b_{n}\right\} $ expressed in
terms of polynomials $\left\{ a_{n}\right\} .$ Like in the setting of
Corollary \ref{rozw} they would be expressed in terms of products of
coefficients $\left\{ \lambda _{n}\right\} $ (in fact either only with odd
or even numbers) and consequently obtain expansion similar (\ref{szereg})
(in fact involving polynomials $\left\{ b_{n}\right\} $ with even numbers).
\end{remark}

\subsection{Open problems}

\begin{itemize}
\item Is it possible to simplify equation (\ref{eqn2}) and reduce it to the
first order recursive equation?

\item Is it possible to simplify system of equations (\ref{s1}-\ref{s4}) and
reduce it to the problem of solving system of the first order recursive
equations?

\item More generally is it possible to solve general system of equations
presented in Proposition 1 of \cite{Szab2010} or at least deduce more
properties of coefficients $\left\{ c_{n}^{\left( j\right) }\right\} _{n\geq
1,1\leq j\leq r}$ not only that for $j>r$ they are zeros.

\item Is it possible give some properties of connection coefficients between
the two sets of orthogonal polynomials given the fact that orthogonalizing
measures are related by the known relationship $dA\left( x\right)
\allowbreak =\allowbreak F\left( x\right) dB\left( x\right) $ for functions $%
F$ different from the reciprocal of a polynomial. It seems possible to
consider rational functions to start generalization.
\end{itemize}

\section{Proofs\label{dowody}}

\begin{proof}[Proof of Theorem \protect\ref{glowny}]
Noticing that the proof of Proposition 1, iii) does not require the measures 
$dA\left( x\right) $ and $dB\left( x\right) $ to have densities we can apply
its assertion and deduce that if positive, normalized measures are related
by the relationship (\ref{jednomian}) then the polynomial sequences $\left\{
a_{n}\right\} $ and $\left\{ b_{n}\right\} $ orthogonal respectively with
respect to these measures are related by (\ref{zaleznosc}). Hence sequence $%
\left\{ \kappa _{n}\right\} $ exists and consequently we have $a_{n}\left(
x\right) \allowbreak =\allowbreak b_{n}\left( x\right) +\kappa
_{n}b_{n-1}\left( x\right) .$ Remembering that sequences of polynomials $%
\left\{ a_{n}\right\} $ and $\left\{ b_{n}\right\} $ are orthogonal and
satisfy the following 3-term recurrences: 
\begin{eqnarray*}
a_{n+1}\left( x\right) \allowbreak &=&\allowbreak (x-\alpha _{n})a_{n}\left(
x\right) -\hat{\alpha}_{n-1}a_{n-1}\left( x\right) , \\
b_{n+1}\left( x\right) &=&\left( x-\beta _{n}\right) b_{n}\left( x\right) -%
\hat{\beta}_{n-1}b_{n-1}\left( x\right) .
\end{eqnarray*}%
So on one hand we have 
\begin{eqnarray*}
xa_{n}\left( x\right) \allowbreak &=&\allowbreak a_{n+1}\left( x\right)
+\alpha _{n}a_{n}\left( x\right) +\hat{\alpha}_{n-1}a_{n-1}\allowbreak \\
&=&\allowbreak b_{n+1}\left( x\right) \allowbreak +\allowbreak (\kappa
_{n+1}+\alpha _{n}\kappa _{n})b_{n}\allowbreak +\allowbreak (\alpha
_{n}\kappa _{n}+\hat{\alpha}_{n-1})b_{n-1}\left( x\right) \allowbreak
+\allowbreak \hat{\alpha}_{n-1}\kappa _{n-1}b_{n-2}\left( x\right) .
\end{eqnarray*}%
On the other we have:%
\begin{eqnarray*}
xa_{n}\left( x\right) \allowbreak &=&\allowbreak xb_{n}\left( x\right)
+x\kappa _{n}b_{n-1}\left( x\right) \allowbreak \\
&=&\allowbreak b_{n+1}\left( x\right) \allowbreak +\allowbreak \left( \beta
_{n}+\kappa _{n}\right) b_{n}\left( x\right) \allowbreak +\allowbreak (\hat{%
\beta}_{n-1}+\kappa _{n}\beta _{n-1})b_{n-1}\allowbreak +\allowbreak \kappa
_{n}\hat{\beta}_{n-2}b_{n-2}\left( x\right) .
\end{eqnarray*}%
Hence we must have:%
\begin{eqnarray*}
\kappa _{n+1}+\alpha _{n} &=&\beta _{n}+\kappa _{n}, \\
\alpha _{n}\kappa _{n}+\hat{\alpha}_{n-1} &=&\hat{\beta}_{n-1}+\kappa
_{n}\beta _{n-1}, \\
\hat{\alpha}_{n-1}\kappa _{n-1} &=&\kappa _{n}\hat{\beta}_{n-2}
\end{eqnarray*}%
Now let us get $\alpha _{n}$ from the first of the equations 
\begin{equation*}
\alpha _{n}\allowbreak =\allowbreak \beta _{n}+\kappa _{n}-\kappa _{n+1}.
\end{equation*}%
We get further 
\begin{equation*}
\hat{\alpha}_{n-1}\allowbreak =\allowbreak -\beta _{n}\kappa _{n}-\kappa
_{n}^{2}+\kappa _{n}\kappa _{n+1}+\hat{\beta}_{n-1}+\kappa _{n}\beta _{n-1}.
\end{equation*}%
So finally we have:%
\begin{equation*}
\kappa _{n-1}(-\beta _{n}\kappa _{n}-\kappa _{n}^{2}+\kappa _{n}\kappa
_{n+1}+\hat{\beta}_{n-1}+\kappa _{n}\beta _{n-1})\allowbreak =\allowbreak
\kappa _{n}\hat{\beta}_{n-2}
\end{equation*}%
dividing both sides by $\kappa _{n}\kappa _{n-1}$ we get:%
\begin{equation*}
\kappa _{n+1}=\kappa _{n}+\frac{\hat{\beta}_{n-2}}{\kappa _{n-1}}-\frac{\hat{%
\beta}_{n-1}}{\kappa _{n}}+\beta _{n}-\beta _{n-1}
\end{equation*}%
Now notice that we can rearrange terms on both sides of this equation in the
following way:%
\begin{equation*}
\kappa _{n+1}+\frac{\hat{\beta}_{n-1}}{\kappa _{n}}-\beta _{n}=\kappa _{n}+%
\frac{\hat{\beta}_{n-2}}{\kappa _{n-1}}-\beta _{n-1},
\end{equation*}%
proving that quantity $\kappa _{n}+\frac{\hat{\beta}_{n-2}}{\kappa _{n-1}}%
-\beta _{n-1}$ does not depend on $n$ and is equal to $\kappa
_{2}\allowbreak +\allowbreak \frac{\hat{\beta}_{0}}{\kappa _{1}}\allowbreak
-\allowbreak \beta _{1}.$

We can easily find this quantity by finding directly quantities $\kappa _{1}$
and $\kappa _{2}.$

Naturally we have $\kappa _{0}=1$. Remembering that $dB\left( x\right)
=\allowbreak (d+cx)dA\left( x\right) ,$ that 
\begin{equation*}
b_{n+1}\left( x\right) =\left( x-\beta _{n}\right) b_{n}\left( x\right) -%
\hat{\beta}_{n}b_{n-1}\left( x\right)
\end{equation*}%
and since $\int_{\limfunc{supp}A}a_{1}\left( x\right) dA\left( x\right)
\allowbreak =\allowbreak 0$ we must have 
\begin{equation*}
1\allowbreak =\int_{\limfunc{supp}A}dA\left( x\right) \allowbreak
=\allowbreak \int_{\limfunc{supp}A}\frac{C}{\left( D+x\right) }dB\left(
x\right) .
\end{equation*}%
Now since $a_{1}\left( x\right) \allowbreak =\allowbreak b_{1}\left(
x\right) \allowbreak +\allowbreak \kappa _{1}$ we have\ 
\begin{eqnarray*}
0\allowbreak &=&\allowbreak \int_{\limfunc{supp}A}\frac{C\left( b_{1}\left(
x\right) +\kappa _{1}\right) }{\left( D+x\right) }dB\left( x\right)
\allowbreak \\
&=&\allowbreak C+(\kappa _{1}-\beta _{0}-D)\int \frac{C}{x+D}dB\left(
x\right) \allowbreak =\allowbreak C+\kappa _{1}-\beta _{0}-D,
\end{eqnarray*}%
So%
\begin{equation*}
\kappa _{1}\allowbreak =\allowbreak \beta _{0}+D-C.
\end{equation*}%
To find $\kappa _{2}$ we use the fact that $a_{2}\left( x\right) \allowbreak
=\allowbreak b_{2}\left( x\right) +\kappa _{2}b_{1}\left( x\right) .$ Hence
we have: 
\begin{eqnarray*}
0\allowbreak &=&\allowbreak C\int_{\limfunc{supp}A}\frac{\left( b_{2}\left(
x\right) +\kappa _{2}b_{1}\left( x\right) \right) }{\left( D+x\right) }%
dB\left( x\right) \allowbreak \\
&=&\allowbreak \allowbreak C\int_{\limfunc{supp}A}\frac{\left( (-D-\beta
_{1}+\kappa _{2})(b_{1}\left( x\right) +\kappa _{1}-\kappa _{1})-\hat{\beta}%
_{0}\right) }{\left( D+x\right) }dB\left( x\right) \allowbreak \\
&=&\allowbreak C\int_{\limfunc{supp}A}\frac{\left( (D+\beta _{1}-\kappa
_{2})\kappa _{1}-\hat{\beta}_{0}\right) }{\left( D+x\right) }dB\left(
x\right) \allowbreak =\allowbreak \left( (D+\beta _{1}-\kappa _{2})\kappa
_{1}-\hat{\beta}_{0}\right) .
\end{eqnarray*}%
Hence we see that $\kappa _{2}\allowbreak +\allowbreak \frac{\hat{\beta}_{0}%
}{\kappa _{1}}\allowbreak -\allowbreak \beta _{1}\allowbreak =\allowbreak D.$
\end{proof}

\begin{proof}[Proof of Proposition \protect\ref{uog}]
Assuming that both sequences of polynomials i.e. $\left\{ a_{n}\right\} $
and $\left\{ b_{n}\right\} $ are orthogonal we have on one hand: 
\begin{eqnarray*}
xa_{n}\left( x\right) \allowbreak &=&\allowbreak a_{n+1}+\alpha
_{n}a_{n}\left( x\right) +\hat{\alpha}_{n-1}a_{n-1}\left( x\right)
\allowbreak =\allowbreak \\
&=&\allowbreak b_{n+1}\left( x\right) \allowbreak +\allowbreak \left( \kappa
_{n+1}+\alpha _{n}\right) b_{n}\left( x\right) \allowbreak +\allowbreak
\left( \lambda _{n+1}+\alpha _{n}\kappa _{n}+\hat{\alpha}_{n-1}\right)
b_{n-1}\left( x\right) \allowbreak \\
&&+\allowbreak \left( \alpha _{n}\lambda _{n}+\hat{\alpha}_{n-1}\kappa
_{n-1}\right) b_{n-2}\left( x\right) \allowbreak +\allowbreak \hat{\alpha}%
_{n-1}\lambda _{n-1}b_{n-3}\left( x\right) .
\end{eqnarray*}%
and on the other:%
\begin{eqnarray*}
xa_{n}\left( x\right) \allowbreak &=&\allowbreak x\left( b_{n}\left(
x\right) +\kappa _{n}b_{n-1}\left( x\right) +\lambda _{n}b_{n-2}\left(
x\right) \right) \allowbreak \allowbreak \\
&=&\allowbreak b_{n+1}\allowbreak +\allowbreak \left( \beta _{n}+\kappa
_{n}\right) b_{n}\left( x\right) \allowbreak +\allowbreak \left( \hat{\beta}%
_{n-1}+\kappa _{n}\beta _{n-1}+\NEG{\lambda}_{n}\right) b_{n-1}\left(
x\right) \allowbreak +\allowbreak \\
&&\left( \kappa _{n}\hat{\beta}_{n-2}+\lambda _{n}\beta _{n-2}\right)
b_{n-2}\left( x\right) +\allowbreak \lambda _{n}\hat{\beta}%
_{n-3}b_{n-3}\left( x\right) .
\end{eqnarray*}

Comparing expressions by $b_{n},$ $b_{n-1},$ $b_{n-2}$ and $b_{n-3}$ we get
equations (\ref{s1}-\ref{s4}).
\end{proof}

\end{document}